\documentclass{article}

\usepackage{url}
\usepackage{amsthm}
\usepackage{amsmath}
\usepackage{amssymb}
\newtheorem{Lem}{Lemma}

\newtheorem{theorem}{Theorem}

\numberwithin{equation}{section}
\title{A new upper bound for $|\zeta(1+ it)|$}

\author{Timothy Trudgian\footnote{Supported by ARC Grant DE120100173.}\\
Mathematical Sciences Institute\\ The Australian National University,
 ACT 0200, Australia\\ timothy.trudgian@anu.edu.au
}
\begin{document}
\maketitle
\section*{Abstract}
It is known that $|\zeta(1+ it)|\ll (\log t)^{2/3}$. This paper provides a new explicit estimate, viz.\ $|\zeta(1+ it)|\leq \frac{3}{4} \log t$, for $t\geq 3$. This gives the best upper bound on $|\zeta(1+ it)|$ for $t\leq 10^{2\cdot 10^{5}}$.

\section{Introduction}
Mellin \cite{Mellin} (see also \cite[Thm 3.5]{Titchmarsh}) was the first to show that 
\begin{equation}\label{Land}
\zeta(1+ it) \ll \log t.
\end{equation} 
This was improved by Littlewood (see, e.g.,\ \cite[Thm 5.16]{Titchmarsh}) to
\begin{equation}\label{Lit}
\zeta(1+ it)\ll \frac{\log t}{\log\log t}.
\end{equation} 
This was improved by several authors; the best known\footnote{As usual, the Riemann hypothesis gives a stronger result, viz., $\zeta(1+ it) \ll \log\log t$ (see, e.g.,\ \cite[\S 14.18]{Titchmarsh}).} result (see, e.g.\ \cite[(6.19.2)]{Titchmarsh}) is
\begin{equation}\label{Vin}
\zeta(1+ it)\ll (\log t)^{2/3}.
\end{equation}

Insofar as explicit results are concerned, Backlund \cite{Backlund}  made (\ref{Land}) explicit by proving that 
\begin{equation}\label{Back}
|\zeta(1+ it)|\leq \log t,
\end{equation}
for $t\geq 50$. Ford \cite{Ford} has made (\ref{Vin}) explicit by proving that 
\begin{equation}\label{Ford}
|\zeta(1+ it)| \leq 72.6 (\log t)^{2/3},
\end{equation}
for $t\geq 3$. Ford's result is actually much more general: he obtains excellent bounds for $|\zeta(\sigma + it)|$ where $\sigma$ is near $1$. Should one be interested in a bound only on $\sigma =1$, one can improve on (\ref{Ford}) to show\footnote{The integral inequality on \cite[p.\ 622]{Ford}, originally verified for $y\geq 0,$ can now be evaluated at $y=0$ only.} that $|\zeta(1+ it)|\leq 62.6 (\log t)^{2/3}$. Note that this improves on (\ref{Back}) when $t\geq 10^{10^{5}}$. Without a complete overhaul of Ford's paper it seems unlikely that his methods could furnish a bound superior to (\ref{Back}) when $t$ is at all modest, say $t\leq 10^{100}$.

To the knowledge of the author there is no explicit bound of the form (\ref{Lit}). One could follow the arguments of \cite[\S 5.16]{Titchmarsh} to produce such a bound, though this leads to a result that only improves on (\ref{Back}) when $t$ is astronomically large. However one can still use the ideas in \cite[\S 5.16]{Titchmarsh} to reprove (\ref{Land}). Indeed if one were lucky, as the author was, one may even be able to supersede (\ref{Back}). This fortune is summarised in the following theorem.
\begin{theorem}\label{onlyt}
\begin{equation*}
|\zeta(1+ it)|\leq \tfrac{3}{4}\log t,
\end{equation*}
when $t\geq 3$.
\end{theorem}

Good explicit bounds on $|\zeta(1+ it)|$ enable one to bound the zeta-function more effectively throughout the critical strip. Indeed Theorem \ref{onlyt} can be used to improve the estimate on $S(T)$ given in \cite{TrudS2}.
\section{Backlund's result}
To prove (\ref{Back}) consider $\sigma >1$ and $t>1$, and write $\zeta(s) - \sum_{n\leq N} n^{-s} = \sum_{N<n} n^{-s}$. Now invoke the following version of the Euler--Maclaurin summation formula --- this can be found in \cite[Thm 2.19]{Murty}.
\begin{Lem}\label{EML}
Let $k$ be a nonnegative integer and $f(x)$ be $(k+1)$ times differentiable on the interval $[a, b]$. Then
\begin{equation*}\label{Ems}
\begin{split}
\sum_{a<n\leq b} f(n) &= \int_{a}^{b} f(t)\, dt + \sum_{r=0}^{k} \frac{(-1)^{r+1}}{(r+1)!} \left\{f^{(r)}(b) - f^{(r)}(a)\right\} B_{r+1}\\
&+ \frac{(-1)^{k}}{(k+1)!} \int_{a}^{b} B_{k+1}(x) f^{(k+1)}(x)\, dx,
\end{split}
\end{equation*}
where $B_{j}(x)$ is the $j$th periodic Bernoulli polynomial and $B_{j} = B_{j}(0)$.
\end{Lem}
Apply this to $f(n) = n^{-s}$, with $k=1$, $a = N$ and with $b$ dispatched to infinity. Thus
\begin{equation}\label{frs}
\zeta(s) - \sum_{n\leq N-1} n^{-s}  = \frac{N^{1-s}}{s-1} + \frac{1}{2N^{s}} + \frac{s}{12N^{s+1}} - \frac{s(s+1)}{2}\int_{N}^{\infty} \frac{\{x\}^{2} - \{x\} + \frac{1}{6}}{x^{s+2}}\, dx,
\end{equation}
where, since the right-side converges for $\Re(s) > -1$, the equation remains valid when $s=1+ it$. Hence one can estimate the sum in (\ref{frs}) using
\begin{equation}\label{logs}
\sum_{n\leq N} \frac{1}{n} \leq \log N + \gamma + \frac{1}{N},
\end{equation}
which follows from partial summation, and in which $\gamma$ denotes Euler's constant. Now if $N = [t/m]$, where $m$ is a positive integer to be chosen later, (\ref{frs}) and (\ref{logs}) combine to show that
\begin{equation}\label{mlast}
|\zeta(1+ it)| - \log t\leq  - \log m + \gamma + \frac{1}{t} + \frac{m}{2(t-m)} + \frac{m^{2}(1+t)(4+t)}{24(t-m)^{2}}.
\end{equation}
The aim is to choose $m$ and $t_{0}$ such that $t\geq t_{0}$ guarantees the right-side of (\ref{mlast}) to be negative. It is easy to verify that when $m=3$, choosing $t= 49.385\ldots$ suffices. Thus (\ref{Back}) is true for all $t\geq 50$; a quick computation shows that (\ref{Back}) remains true for $t\geq 2.001\ldots$.

It seems impossible to improve upon (\ref{Back}) without a closer analysis of sums of the form $\sum_{a<n\leq 2a} n^{-it}$. Taking further terms in the Euler--Maclaurin expansion in (\ref{frs}) does not achieve an overall saving; choosing $N = [t^{\alpha}]$ for some $\alpha<1$ in (\ref{logs}) means that the integral in (\ref{frs}) is no longer bounded.

The next section aims at securing a good bound for $\sum_{a< n \leq 2a}n^{-it}$ for `large' values of $a$. For `small' values of $a$ one may estimate the sum trivially. The inherent optimism is that, when combined, these two estimates give an improvement on (\ref{Back}).

\section{Exponential sums: beyond Backlund}
The following is an explicit version of Theorem 5.9 in \cite{Titchmarsh}.
\begin{Lem}[Cheng and Graham]\label{CLem}
Assume that $f(x)$ is a real-valued function with two continuous derivatives when $x\in(a, 2a]$. If there exist two real numbers $V<W$ with $W>1$ such that
\begin{equation*}
\frac{1}{W}\leq |f''(x)|\leq \frac{1}{V}
\end{equation*}
for $x\in[a+1, 2a]$, then
\begin{equation*}
\bigg|\sum_{a<n\leq 2a} e^{2\pi f(n)}\bigg| \leq \frac{1}{5}\left(\frac{a}{V} + 1\right)(8W^{1/2} + 15).
\end{equation*}
\end{Lem}
\begin{proof}
See Lemma 3 in \cite{Cheng}.
\end{proof}
Applying Lemma \ref{CLem} to $f(x) = -(2\pi)^{-1} t \log x$ gives
\begin{equation}\label{preA1}
\bigg|\sum_{a<n\leq 2a} n^{-it}\bigg| \leq t^{1/2}\left\{ \frac{8}{5} \sqrt{\frac{2}{\pi}} + \frac{16\sqrt{2\pi}a}{5t} + \frac{3t^{1/2}}{2\pi a} + 3t^{-1/2}\right\},
\end{equation}
subject\footnote{This is to ensure that, in Lemma \ref{CLem}, $W>1$ --- see (\ref{footy}).} to $8\pi a^{2} >t$. Now take\footnote{To ensure that this is a non-empty interval see (\ref{footy}).} $A_{1} t^{1/2} < a \leq [t/m]$ for some constant $A_{1}$ and positive integer $m$ to be determined later. If $t\geq t_{0}$ then (\ref{preA1}) shows that
\begin{equation*}\label{A2def}
\bigg|\sum_{a<n\leq 2a} n^{-it}\bigg| \leq A_{2} t^{1/2},
\end{equation*}
and hence, by partial summation,
\begin{equation}\label{Longs}
\bigg|\sum_{a<n\leq 2a} n^{-1-it}\bigg| \leq A_{2} a^{-1} t^{1/2} \leq \frac{A_{2}}{A_{1}},
\end{equation}
where
\begin{equation*}
A_{2} = \frac{8}{5} \sqrt{\frac{2}{\pi}} + \frac{16\sqrt{2\pi}}{5m} + \frac{3}{2\pi A_{1}} + 3t_{0}^{-1/2}.
\end{equation*}
One may now apply (\ref{Longs}) to each of the sums on the right-side of 
\begin{equation*}
\bigg|\sum_{A_{1} t^{1/2} < n \leq (t/m)}\frac{1}{n^{1+ it}}\bigg| = \sum_{\frac{1}{2}(t/m) < n \leq (t/m)} + \sum_{\frac{1}{4}(t/m) < n \leq \frac{1}{2}(t/m)} + \cdots.
\end{equation*}
There are at most
\begin{equation}\label{nos}
\frac{\frac{1}{2} \log t - \log (mA_{1}) + \log 2}{\log 2}
\end{equation}
such sums.
This gives an upper bound for $\sum n^{-1-it}$ when $n> A_{1}t^{1/2}$. When $n \leq A_{1}t^{1/2}$ one may use (\ref{logs}) to estimate the sum trivially. 

\section{Proof of Theorem \ref{onlyt}}
In $\zeta(s) - \sum_{n\leq N} n^{-s} = \sum_{N<n} n^{-s}$ use Euler--Maclaurin summation (Lemma \ref{EML}) to $k$ terms. Choosing $N-1 = [t/m]$, recalling (\ref{Longs}) and (\ref{nos}), and estimating all complex terms trivially gives
\begin{equation}\label{Emss}
\begin{split}
|\zeta(1+ it)| &\leq  \log t\left\{ \frac{1}{2} + \frac{A_{2}}{2 A_{1} \log 2}\right\} + \frac{A_{2}\{\log 2 - \log(mA_{1})\}}{A_{1} \log 2} + \log A_{1} + \gamma \\
&+ \frac{1}{A_{1} t_{0}^{1/2}} +  \frac{m}{2t} + \frac{1}{t} + \sum_{r=1}^{k} \frac{|B_{r+1}|}{(r+1)!} (1+t)\cdots(r+t)\left(\frac{m}{t}\right)^{r+1}\\
&+ \frac{(1+ t)\cdots(k+1+t)}{(k+1)\cdot(k+1)!} \max |B_{k+1}(x)| \left(\frac{m}{t}\right)^{k+1}.
\end{split}
\end{equation}
Note that each term in the $r$-sum in (\ref{Emss}) is $O_{m, k}(t^{-1})$. This is cheap relative to the last term which is $O_{m, k}(1)$. Thus one can take $k$ somewhat large to reduce the burden of the final term. For a given $t_{0}$, when $t\geq t_{0}$ one can optimise (\ref{Emss}) over $k$, $m$ and $A_{1}$ subject to
\begin{equation}\label{footy}
A_{1} > \frac{1}{\sqrt{8\pi}}, \quad m A_{1}\leq t_{0}^{1/2}.
\end{equation}
One finds that, when $k=14, m=6, A_{1} = 23$ then $|\zeta(1+ it)|\leq 0.749818\ldots$, for all $t\geq 10^{8}$.  A numerical check on \textit{Mathematica} suffices to extend the result to all $t\geq 2.391\ldots$, whence Theorem \ref{onlyt} follows.

\subsection{Improvements}
Lemma \ref{CLem} is unable to furnish a value less than $\frac{1}{2}$ in Theorem \ref{onlyt}. On the other hand, by verifying that $|\zeta(1+ it)|< \frac{1}{2} \log t$ for $t$ larger than $10^{8}$ one will improve slightly on Theorem \ref{onlyt}.

One could also take an analogue of Lemma \ref{CLem} that incorporates higher derivatives. Such a result, giving explicit bounds on exponential sums of a function involving $k$ derivatives, is given in \cite[Prop.\ 8.2]{GranvilleRamare}. It is unclear how much could be gained from pursuing this idea.
\bibliographystyle{plain}
\bibliography{themastercanada}

\end{document}